\documentclass[11pt]{amsart}
\usepackage{amssymb}
\usepackage{amssymb, amsmath}
\usepackage{verbatim}
\let\eps=\varepsilon
\let\wt=\widetilde

\def\cC{{\mathcal C}}

\def\cF{{\mathcal F}}

\def\cL{{\mathcal L}}

\def\cP{{\mathcal P}}
\def\cQ{{\mathcal Q}}

\def\cS{{\mathcal S}}

\def\N{{\mathbb N}}
\def\R{{\mathbb R}}

\def\Z{{\mathbb Z}}

\def\Supp{\, \mbox{Supp}\,  }

\def\virgp{\raise 2pt\hbox{,}}
\def\cdotpv{\raise 2pt\hbox{;}}

\def\div{ \hbox{\rm div}\,  }

\def\ddj{\dot \Delta_j}
\def\ddk{\dot \Delta_k}

\newtheorem{defi}{Definition}[section]
\newtheorem{thm}{Theorem}[section]
\newtheorem{lem}{Lemma}[section]
\newtheorem{rmk}{Remark}[section]
\newtheorem{cor}{Corollary}[section]

\newcommand{\ben}{\begin{eqnarray}}
\newcommand{\een}{\end{eqnarray}}
\newcommand{\beno}{\begin{eqnarray*}}
\newcommand{\eeno}{\end{eqnarray*}}
\begin{document}
\title[]{Compressible Navier-Stokes system~: large solutions and incompressible limit}
\author[R. Danchin]{Rapha\"{e}l Danchin}
\address[R. Danchin]{Universit\'{e} Paris-Est,  LAMA (UMR 8050), UPEMLV, UPEC, CNRS, Institut Universitaire de France,
 61 avenue du G\'{e}n\'{e}ral de Gaulle, 94010 Cr\'{e}teil Cedex, France.} \email{raphael.danchin@u-pec.fr}
\author[P.B. Mucha]{Piotr Bogus\l aw Mucha}
\address[P.B. Mucha]{Instytut Matematyki Stosowanej i Mechaniki,
 Uniwersytet Wars\-zawski, 
ul. Banacha 2,  02-097 Warszawa, Poland.} 
\email{p.mucha@mimuw.edu.pl}

\begin{abstract} 
Here we prove the existence of global in time regular solutions to the two-dimensional compressible Navier-Stokes equations 
supplemented with arbitrary large initial velocity $v_0$ and almost  constant
density $\varrho_0$, for  large volume (bulk) viscosity. The result is generalized to the higher 
dimensional case under the additional assumption that the strong solution of the classical 
incompressible Navier-Stokes equations supplemented with the divergence-free
projection of $v_0,$ is global.  The systems are examined in $\R^d$
with $d \geq 2$,  in the critical $\dot B^s_{2,1}$ Besov spaces  framework.
\end{abstract}
\maketitle

\section{Introduction}

We are concerned with  the following compressible Navier-Stokes equations in the whole space $\R^d$:
\begin{equation}\label{eq:NSC}
\left\{ \begin{array}{l}
  \varrho_t +\div (\varrho v) =0,\\[5pt]
\varrho v_t + \varrho v \cdot \nabla v - \mu \Delta v - (\lambda + \mu) \nabla \div v + \nabla P =0,
 \end{array}\right.
\end{equation}
supplemented with  initial data: $\varrho|_{t=0} =\varrho_0$ and $v|_{t=0}=v_0$.
\medbreak
The  pressure function $P$ is given and assumed to be strictly increasing. The shear and volume viscosity coefficients
$\lambda$ and $\mu$ are  constant (just for simplicity) and fulfill the standard  strong parabolicity assumption:
\begin{equation}\label{eq:viscous}
\mu>0\quad\hbox{and}\quad
\nu:=\lambda+2\mu>0.
\end{equation}


Starting with the pioneering work by  Matsumura and Nishida \cite{MN80,MN83} in the
beginning of the eighties,  a number of papers have been dedicated to the challenging issue
of proving the global existence of strong solutions for \eqref{eq:NSC} 
in different contexts (whole space or domains, dimension $d=2$ or $d\geq3$, and so on). 
One may mention in particular the works by Zaj\c aczkowski \cite{VZ}, Shibata \cite{KS}, Danchin \cite{D1}, Mucha \cite{M,MZ02,MZ} and, more  recently,  by Kotschote \cite{Kot1,Kot2}.
The common point between all those papers is that  the initial velocity is  assumed to be small, and 
that the initial density is close to a stable constant steady state. 

Our main goal is to prove the global existence of strong solutions to \eqref{eq:NSC}
for a class of large initial data. 
  In the two-dimensional case, 
we establish that, indeed, for  fixed shear viscosity $\mu$ and any   initial velocity-field $v_0$ (with critical regularity), the solution to
\eqref{eq:NSC} is global if $\lambda$ is sufficiently large, and $\varrho_0$ sufficiently close (in terms of $\lambda$) 
to some positive constant (say $1$ for notational simplicity). This  result will strongly  rely 
on the fact that, at least formally, the limit velocity for $\lambda\to+\infty$ satisfies the incompressible Navier-Stokes equations:
\begin{equation}\label{NS}
\left\{ \begin{array}{lcr}
  V_t + V\cdot \nabla V - \mu \Delta V+\nabla\Pi=0 & \mbox{in}&  \R_+\times\R^d, \\
  \div V=0 & \mbox{in}& \R_+\times\R^d, \\
  V|_{t=0} = V_0 & \mbox{at} & \R^d,
 \end{array}\right.
\end{equation}
with $V_0$ being the Leray-Helmholtz projection of $v_0$ on divergence-free vector-fields.
\medbreak
We are also interested in similar results  in  dimension $d\geq3.$ 
 However, as  in that case the global existence issue of strong solutions for \eqref{NS}  supplemented 
 with general data is open,  we have  to assume first that $V_0$ generates  a global strong solution to   \eqref{NS}, and then
to  analyze the stability of that solution in the setting of the compressible model \eqref{eq:NSC}
with large $\lambda.$
\medbreak

%
%
%

Let us first consider the two-dimensional case, assuming  that  initial data $\varrho_0$ and $v_0$ fulfill 
the critical regularity assumptions of \cite{D1}, namely\footnote{The reader may refer to the next 
section for the definition of homogeneous Besov spaces $\dot B^s_{2,1}(\R^d).$}
$$
a_0:=(\varrho_0-1)\in  \dot B^{0}_{2,1}\cap \dot B^{1}_{2,1}(\R^2)\quad\hbox{and}\quad
v_0\in  \dot B^{0}_{2,1}(\R^2).
$$
Then  the initial data   $V_0$ of \eqref{NS} is in $\dot B^{0}_{2,1}(\R^2).$ Therefore,
in light of  the well-known embedding  $\dot B^0_{2,1}(\R^2) \subset L_2(\R^2), $ we are guaranteed
that  it  generates a unique global solution $V$ 
in the energy  class $$V^{1,0}(\R^2 \times \R_+):= \cC_{b}(\R_+;L_2(\R^2)) \cap L_2(\R_+;\dot H^1(\R^2)),$$
that satisfies  the energy identity:
$$
\|V(t)\|_{L_2}^2+2\mu\int_0^t\|\nabla V\|_{L_2}^2\,d\tau=\|V_0\|_{L_2}^2.
$$
Based on that fact, one may prove that the additional regularity of  $V_0$ is  preserved 
through the time evolution (see Theorem \ref{Th:NS} in the  Appendix), that is
\begin{equation}\label{V-bes}
 V \in \cC_b(\R_+;\dot B^0_{2,1}(\R^2)) \cap L_1(\R_+;\dot B^2_{2,1}(\R^2)).
\end{equation}
Let us now state  our main existence result for \eqref{eq:NSC} in the two-dimensional setting.
\begin{thm}\label{th:2D}
Let $\mu \leq \nu$. 
 Let $v_0 \in \dot B^0_{2,1}(\R^2)$ and $\varrho_0$ such that $a_0:=(\varrho_0-1)\in\dot B^0_{2,1}\cap \dot B^1_{2,1}(\R^2).$
 There exists a large constant $C$ such that for $V_0=\cP v_0$, the divergence-free part of the initial velocity, we set 
 \begin{equation}\label{def:M} M:= C\|V_0\|_{\dot B^0_{2,1}}\exp\Bigl(\frac C{\mu^4}\|V_0\|_{L_2}^4\Bigr)\end{equation}
and if  $\nu$ satisfies
$$
Ce^{CM}\bigl( \|a_0\|_{\dot B^0_{2,1}}+\nu\|a_0\|_{\dot B^1_{2,1}}+\|\cQ u_0\|_{\dot B^0_{2,1}}+M^2
 +\mu^2 \bigr)\leq \sqrt\nu\,\sqrt{\mu},
 $$
 where $\cQ$ stands for the projection operator on potential vector-fields,
then there exists a unique global in time regular solution $(\varrho,v)$ to \eqref{eq:NSC} such that
\begin{equation}\label{eq:reg2D}
\begin{array}{c}
   v \in \cC_b(\R_+;\dot B^0_{2,1}),\qquad v_t,\nabla^2v\in L_1(\R_+;\dot B^0_{2,1}),\\[1ex]
  a:=(\varrho -1) \in\cC(\R_+; \dot B^0_{2,1}\cap \dot B^1_{2,1}) \cap
  L_2(\R_+;\dot B^1_{2,1}).\end{array}\end{equation}
In addition, the following bound is fulfilled by the solution:
$$
\displaylines{
\quad \|\cQ v\|_{L_\infty(\R_+;\dot B^0_{2,1})}+
  \|\cQ v_t, \nu\nabla^2\cQ v\|_{L_1(\R_+;\dot B^0_{2,1})}
    + \|a\|_{L_\infty(\R_+;\dot B^0_{2,1})} +
      \nu\|a\|_{L_\infty(\R_+;\dot B^1_{2,1})}\hfill\cr\hfill
     +\nu^{1/2}\bigl(\|\cP v -V\|_{L_\infty(\R_+;\dot B^0_{2,1})}+
  \|\cP v_t -V_t, \mu\nabla^2 (\cP v -V)\|_{L_1(\R_+;\dot B^0_{2,1})}\bigr) 
            \hfill\cr\hfill
    \leq   Ce^{CM}\bigl(\|a_0\|_{\dot B^0_{2,1}} +\nu \|a_0\|_{\dot B^1_{2,1}}+\|\cQ u_0\|_{\dot B^0_{2,1}} +M^2+\mu^2 \bigr).
        \quad} $$
\end{thm}

Let us emphasize that in contrast with the global existence results cited above, we do not require any smallness 
condition on  the initial velocity : the volume viscosity $\lambda$ just has to be  sufficiently large. 
The mechanism underneath is that having large $\lambda$ provides strong dissipation on the potential part
of the velocity, and thus makes   our flow almost incompressible.
At the same time,  strong dissipation does not  involve the divergence free part of the flow, but as
we are in dimension two, it is known that it generates a global strong incompressible solution.
In fact, our statement may be seen as   a  stability result for incompressible flows 
within compressible flows. 

Our  result  has some similarity with  that  of the first author in \cite{D3,D4} where it is
shown that large initial velocities give rise to global strong solutions in the low Mach number asymptotics. 
However, the mechanism leading to global existence therein strongly relies on the dispersive (or highly oscillating)
properties of the acoustic wave equations. This is  in sharp contrast with the situation we are looking at
here, where dispersion completely disappears when $\lambda\to+\infty.$
\smallbreak
There are also examples of large data generating global strong solutions to the compressible Navier-Stokes
equations, independently  of any  asymptotic  considerations. In this regard, one has to mention the  result by 
 Kazhikov-Weigant \cite{VK}  in the two dimensional case, where it is assumed that  the volume viscosity $\lambda$ has some suitable 
dependence  with respect to the density  the density (like $\lambda(\varrho)=\varrho^\beta$ for some $\beta>3$). 
In contrast, here we do not require any particular nonlinear structure of the viscosity 
coefficients, but rather that the volume viscosity is large enough.
Finally, in a recent paper  \cite{Haspot3} dedicated to the shallow water equations (that is $\mu$ depends
linearly on $\varrho$ and $\lambda=0$), B. Haspot established the existence of global strong solutions
allowing for large potential  part of the initial velocity.
\medbreak
As a  by product of Theorem \ref{th:2D}, we get that 
 $(\varrho,v)\to(1,V)$ with a convergence rate of order $\nu^{-1/2}.$ This is stated more exactly
 in the following corollary. 
\begin{cor}\label{c:2D}
Let $v_0$ be any vector field in $\dot B^0_{2,1}(\R^2),$ and $M$ be defined by \eqref{def:M}. Then for large enough $\nu$ (or equivalently $\lambda$), 
System \eqref{eq:NSC} supplemented with initial density $1$ and initial velocity $v_0$
has  a unique global solution $(\varrho,v)$ in the space given by \eqref{eq:reg2D}.
Furthermore,  if $V$ stands for the solution to \eqref{NS} then  $(\varrho,v)\to(1,V)$ as follows:
    $$     \displaylines{\quad\|\varrho-1\|_{L_\infty(\R_+;\dot B^1_{2,1})}+\|\nabla^2\cQ v\|_{L_1(\R_+;\dot B^0_{2,1})}
+  \|\cP v -V\|_{L_\infty(\R_+;\dot B^0_{2,1})}    \hfill\cr\hfill+
  \|\cP v_t -V_t, \mu\nabla^2 (\cP v -V)\|_{L_1(\R_+;\dot B^0_{2,1})}   \leq C\nu^{-1/2}\,\sqrt{\mu}.\quad}$$
\end{cor}

Let us now describe our main result in the high-dimensional case $d\geq3.$ 
Then it turns out that our approach for exhibiting large global solutions  is essentially the same,
once it is known that the limit system \eqref{NS} supplemented 
with initial data $V_0:=\cP v_0$ has  a global strong solution with suitable regularity. 
However, as constructing such global solutions  in the large data case is still an open question, 
we will   assume \emph{a priori} that  $V_0$ generates a global solution $V$ 
in $\cC_b(\R_+;\dot B^{d/2-1}_{2,1}(\R^d)).$ This only requirement will ensure,  thanks to 
the  result of Gallagher-Iftimie-Planchon in \cite{GIP},  that we have in fact a stronger property, namely 
\begin{equation}\label{V-d-bes}
 V \in \cC_b(\R_+;\dot B^{d/2-1}_{2,1}(\R^d))\quad\hbox{\em and}\quad  V_t,\nabla^2V\in  L_1(\R_+;\dot B^{d/2+1}_{2,1}(\R^d)).
\end{equation}
\begin{thm}\label{th:dD}
 Assume that  $d \geq 3.$ Let $v_0 \in \dot B^{d/2-1}_{2,1}(\R^d)$ and $\varrho_0$ such that $a_0:=(\varrho_0-1)\in\dot B^{d/2-1}_{2,1}\cap \dot B^{d/2}_{2,1}(\R^d).$ Suppose that \eqref{NS} with initial datum $V_0:=\cP v_0$
 generates a unique global solution $V\in\cC_b(\R_+;\dot B^{d/2-1}_{2,1})$  (thus also \eqref{V-d-bes} is fulfilled), and 
 denote
$$
M:= \|V\|_{L_\infty(\R_+;\dot B^{d/2-1}_{2,1})}+\|V_t,\mu\nabla^2V\|_{L_1(\R_+;\dot B^{d/2-1}_{2,1})}.
$$
 There exists a (large) universal constant $C$ such that  if  $\nu$ satisfies
$$
Ce^{CM}\bigl( \|a\|_{\dot B^{d/2-1}_{2,1}}+\nu\|a_0\|_{\dot B^{d/2}_{2,1}}+\|\cQ u_0\|_{\dot B^{d/2-1}_{2,1}}+M^2
 +\mu^2 \bigr)\leq \sqrt\nu\,\sqrt{\mu},
 $$
then   \eqref{eq:NSC} has a unique global-in-time  solution $(\varrho,v)$  such that
 \begin{equation}
  \begin{array}{l}
   v \in \cC_b(\R_+;\dot B^{d/2-1}_{2,1}),\qquad v_t,\nabla^2v\in L_1(\R_+;\dot B^{d/2-1}_{2,1}),\\
  a:=(\varrho -1) \in\cC(\R_+; \dot B^{d/2-1}_{2,1}\cap \dot B^{d/2}_{2,1}) \cap
  L_2(\R_+;\dot B^{d/2}_{2,1}).
  \end{array}
 \end{equation}
In addition, 
$$
\displaylines{
\quad \|\cQ v\|_{L_\infty(\R_+;\dot B^{d/2-1}_{2,1})}+
  \|\cQ v_t, \nu\nabla^2\cQ v\|_{L_1(\R_+;\dot B^{d/2-1}_{2,1})}
    + \|a\|_{L_\infty(\R_+;\dot B^{d/2-1}_{2,1})} +
      \nu\|a\|_{L_\infty(\R_+;\dot B^{d/2}_{2,1})}\hfill\cr\hfill
     +\nu^{1/2}\bigl(\|\cP v -V\|_{L_\infty(\R_+;\dot B^{d/2-1}_{2,1})}+
  \|\cP v_t -V_t, \mu\nabla^2 (\cP v -V)\|_{L_1(\R_+;\dot B^{d/2-1}_{2,1})}\bigr) 
            \hfill\cr\hfill
    \leq   Ce^{CM}\bigl(\|a_0\|_{\dot B^{d/2-1}_{2,1}} +\nu \|a_0\|_{\dot B^{d/2}_{2,1}}+\|\cQ u_0\|_{\dot B^{d/2-1}_{2,1}} +M^2+\mu^2\bigr)
        \quad} $$
and   $(\varrho,v)\to(1,V)$ as follows:
    $$     \displaylines{\quad\|\varrho-1\|_{L_\infty(\R_+;\dot B^{d/2}_{2,1})}+\|\nabla^2\cQ v\|_{L_1(\R_+;\dot B^{d/2-1}_{2,1})}
+  \|\cP v -V\|_{L_\infty(\R_+;\dot B^{d/2-1}_{2,1})}    \hfill\cr\hfill+
  \|\cP v_t -V_t, \mu\nabla^2 (\cP v -V)\|_{L_1(\R_+;\dot B^{d/2-1}_{2,1})}   \leq C\nu^{-1/2}\,\sqrt{\mu}.\quad}$$
\end{thm}

Let us emphasize that even in  dimension $d\geq3,$ there are  examples of large initial data for 
\eqref{NS} generating global smooth solutions. One can refer for instance to \cite{BMN,Gal,M01,PRST} and citations therein.
 Therefore, our second result indeed points out example of large data 
giving rise to global strong solutions for  the compressible system \eqref{eq:NSC}.
\medbreak
Let us finally say a few words on our functional setting. 
Throughout,  we used the so-called critical Besov spaces of type  $\dot B^{s}_{2,1}(\R^2),$  as they are known to provide
essentially the largest class of data for which System \eqref{eq:NSC} may be solved by energy type methods, and 
is well-posed in the sense of Hadamard.
As a matter of fact,  our proof relies on   a suitable  energy method applied to 
the system after localization according to  Littlewood-Paley decomposition (see the definition is the next section).
We believe that it would be possible to derive similar qualitative results in the critical  $L_p$ Besov framework (spaces $\dot B^{s}_{p,1}$). However, we refrained from doing that both because it makes the proof 
 significantly more technical, and because   there are some restrictions to the admissible 
values of $p$ (e.g. $2\leq p<4$ if $d=2$)  so that  the improvement compared to $p=2$ is not so big.
\medbreak 
The rest of the paper unfolds as follows. In the next section,  we introduce Besov spaces and recall basic facts about them. 
Section \ref{s:proof} is devoted to proving both Theorem \ref{th:2D} and \ref{th:dD}. 
In fact, we are able to provide a common proof to  both results as the only difference
between dimension $d=2$ and dimension $d\geq3$ is that we are always guaranteed 
that $V_0$ gives rise to a global regular solution in the former case while it is an additional assumption 
in the latter case.  In Appendix we show Theorem \ref{Th:NS} concerning
the regularity of 2D incompressible flow.


\section{Notation, Besov spaces and basic properties}\label{s:notation}

The \emph{Littlewood-Paley decomposition}  plays a central role in our analysis. 
To define it,   fix some  smooth radial non increasing function $\chi$
supported in the ball $B(0,\frac 43)$ of $\R^d,$ and with value $1$ on, say,   $B(0,\frac34),$ then set
$\varphi(\xi)=\chi(\xi/2)-\chi(\xi).$ We have
$$
\qquad\sum_{j\in\Z}\varphi(2^{-j}\cdot)=1\ \hbox{ in }\ \R^d\setminus\{0\}
\quad\hbox{and}\quad \Supp\varphi\subset \big\{\xi\in\R^d : 3/4\leq|\xi|\leq8/3\big\}\cdotp
$$
The homogeneous dyadic blocks $\ddj$ are defined on tempered distributions by
$$\ddj u:=\varphi(2^{-j}D)u:=\cF^{-1}(\varphi(2^{-j}\cdot)\cF u)=2^{jd}h(2^j\cdot)\star u
\quad\hbox{with}\quad h:=\cF^{-1}\varphi.
$$
In order to ensure that 
\begin{equation}\label{eq:decompo}
f=\sum_{j\in\Z}\ddj f\quad\hbox{in}\quad\cS'(\R^d),
\end{equation}
we restrict our attention to  those tempered distributions $f$ such that
\begin{equation}\label{eq:Sh}
\lim_{k\rightarrow-\infty}\|\dot S_kf\|_{L_\infty}=0,
\end{equation}
where $\dot S_kf$ stands for the low frequency cut-off defined by $\dot S_kf:=\chi(2^{-k}D)f$.
\begin{defi}\label{d:besov}
 For $s\in\R$   the homogeneous Besov space $\dot B^s_{2,1}:=\dot B^{s}_{2,1}(\R^d)$ is  the
set of tempered distributions $f$ satisfying \eqref{eq:Sh} and
$$
\|f\|_{\dot B^s_{2,1}}:=\sum_{j\in\Z} 2^{js}
\|\ddj  f\|_{L_2}<\infty.$$
\end{defi}
\begin{rmk} For $s\leq d/2$ (which is the only case we are concerned with in
this paper), $\dot B^s_{2,1}$ is a Banach space which coincides with 
the completion for $\|\cdot\|_{\dot B^s_{2,1}}$ of the set $\cS_0(\R^d)$ of Schwartz functions
with Fourier transform supported away from the origin. 
\end{rmk}
In many parts of the paper, it will be suitable to split tempered distributions $f$ (e.g. the unknown $a:=\varrho-1$) into  low and high frequencies  as follows:
\begin{equation}\label{eq:lhf}
f^\ell:=\sum_{2^k\nu\leq 1}\ddk f\quad\hbox{and}\quad
f^h:=\sum_{2^k\nu>1}\ddk f.
\end{equation}
The following Bernstein inequalities play an important role in our analysis:
\begin{itemize}\item Direct Bernstein inequality: for all $1\leq p\leq q\leq\infty$ and $k\in\N,$
\begin{equation}
 \|\dot \Delta_j\nabla^k u \|_{L_q(\R^d)} \leq C 2^{j(k+d(\frac{1}{p} - \frac{1}{q}))} \|\dot \Delta_j u\|_{L_p(\R^d)}.
\end{equation}
\item Reverse Bernstein inequality: for all $1\leq p\leq\infty,$ we have
$$ \|\dot \Delta_j u \|_{L_p(\R^d)} \leq C2^{-j}\|\dot \Delta_j\nabla u \|_{L_p(\R^d)}. $$
\end{itemize}
The following  lemma will be needed to estimate the nonlinear terms
of \eqref{eq:NSC} and \eqref{NS}. It is just a consequence of Bony decomposition and of continuity results for the paraproduct and remainder operators, as stated in e.g. Theorem 2.52 of  \cite{BCD}. 
\begin{lem}\label{l:para1}
 Let $g \in \dot B^{s_1}_{2,1}(\R^d)$ and  $h\in\dot B^{s_2}_{2,1}(\R^d)$ for some couple $(s_1,s_2)$ satisfying 
 $$
 s_1\leq d/2,\quad s_2\leq d/2\ \hbox{ and }\ s_1+s_2>0.
 $$
 Then $gh \in \dot B^{s_1+s_2-d/2}_{2,1}(\R^d),$ and we have
\begin{equation}
\|gh\|_{ \dot B^{s_1+s_2-d/2}_{2,1}} \leq C \|g\|_{\dot B^{s_1}_{2,1}}\|h\|_{\dot B^{s_2}_{2,1}}.
\end{equation}
\end{lem}
Finally, let us recall that  any vector field $w=(w^1,\cdots,w^d)$ with components in $\cS'(\R^d)$ satisfying \eqref{eq:Sh}
may be decomposed into one potential part $\cQ w$ and one divergence-free part $\cP w,$ where the projectors $\cP$ 
and $\cQ$ are defined by  
$$
\cQ:=-(-\Delta)^{-1}\nabla\div\quad\hbox{and}\quad \cP:={\rm Id}+(-\Delta)^{-1}\nabla\div.
$$ 
In particular, because $\cP$ and $\cQ$ are smooth homogeneous of degree $0$ Fourier multipliers, 
they map $\dot B^s_{2,1}(\R^d)$ to itself for any $s\leq d/2.$


\section{The proof of the main results}\label{s:proof}

We shall get Theorems \ref{th:2D} and \ref{th:dD}  altogether. In fact, if it is known that the limit system \eqref{NS} 
with initial data $V_0:=\cP v_0$ has a unique solution in our functional  setting, then the proof is 
the same in  any dimension $d\geq2.$ 
The only difference is that  in the 2D case the existence of  a global solution to \eqref{NS} 
is ensured by Theorem \ref{Th:NS} for arbitrary large data whereas, if $d\geq3,$
it is indeed a supplementary assumption.
\medbreak
To simplify the presentation, we  assume from now on that the shear viscosity $\mu$ is  $1.$ This  is not restrictive owing 
to the following change of unknowns and volume viscosity: 
\begin{equation}\label{eq:change}
(\wt\varrho,\wt v)(t,x):= (\varrho,v)(\mu t,\mu x)\quad\hbox{and }\ \wt\lambda=\lambda/\mu.
\end{equation}

We concentrate our attention on  the proof of \emph{global in time} a priori estimates, 
as the local existence issue is nowadays well understood. For example, just assuming that $\varrho_0$ is bounded 
away from zero and that the regularity assumptions of Theorems \ref{th:2D} or \ref{th:dD} are fulfilled, 
Theorem 2 of \cite{D5} provides us with  a unique local solution $(\varrho,v)$ to \eqref{eq:NSC} such that 
\begin{equation}\label{eq:reg}a:=(\varrho-1)\in \cC([0,T);\dot B^{d/2}_{2,1})\quad\hbox{and}\quad
v\in\cC([0,T);\dot B^{d/2-1}_{2,1})\cap L_1(0,T;\dot B^{d/2+1}_{2,1}).\end{equation}
Furthermore,    continuation  beyond $T$ is possible if 
\begin{equation}\label{eq:ex}
\int_0^T\|\nabla v\|_{L_\infty}\,dt<\infty,\quad \|a\|_{L_\infty(0,T;\dot B^{d/2}_{2,1})}<\infty\ \hbox{ and }\ 
\inf_{(t,x)\in[0,T)\times\R^d} \varrho(t,x)>0.
\end{equation}
Finally, from standard results for the transport equation, we see that the additional $\dot B^{d/2-1}_{2,1}$ regularity of
$a$ is preserved through the evolution.
\medbreak
Next, to compare the solutions of  \eqref{eq:NSC} and \eqref{NS},  we set  $u:= v -V.$ {}From the very beginning, the potential $\cQ u$ and divergence-free $\cP u$ parts of the perturbation of the velocity are treated separately. On one hand, because $\cQ u=\cQ v,$
applying $\cQ$ to the velocity equation in \eqref{eq:NSC} yields
$$
(\cQ u)_t+\cQ((1+a)v\cdot\nabla v)-\nu\Delta\cQ u+P'(1+a)\nabla a=-\cQ(a v_t),
$$
hence using $v=\cQ u+\cP u+V$ and assuming $P'(1)=1$ (for notational simplicity), 
\begin{equation}\label{B1}
 (\cQ u)_t + \cQ( (u+V) \cdot \nabla \cQ u) - \nu \Delta \cQ u + \nabla a = -\cQ ( a V_t 
 +a u_t) -\cQ R_2
\end{equation}
with, denoting $k(a):=P'(1+a)-P'(1)=P'(1+a)-1,$ 
\begin{equation}\label{B2}
 R_2= (1+a)(u+V)\cdot\nabla\cP u +(1+a)(u+V)\cdot\nabla V+a(u+V)\cdot\nabla\cQ u+k(a)\nabla a,
\end{equation}
and $a$ satisfying 
\begin{equation}\label{a}
  a_t +\div (au) +\div \cQ u + V \cdot \nabla a=0. 
  \end{equation}
Initial data  are  $\cQ u |_{t=0} =\cQ v_0$ and $a|_{t=0}=a_0.$
\medbreak
On the other hand, applying $\cP$ to the velocity equation of \eqref{eq:NSC} and subtracting the equation for $\cP V=V$ in 
\eqref{NS}, we discover that 
$$\displaylines{
 (\cP u)_t+ \cP((u+V)\cdot \nabla \cP u) - \Delta \cP u = -\cP(a V_t + a u_t)\hfill\cr\hfill
 - \cP\Bigl((1+a)(u+V)\cdot \nabla \cQ u + (1+a) u \cdot \nabla V 
 + a (u+V) \cdot \nabla\cP u +aV\cdot\nabla V\Bigr),} $$
supplemented with the initial datum $\cP u|_{t=0}=0$ (as we assumed $V_0=\cP v_0$).
\medbreak
Note that because  $\cQ u\cdot\nabla\cQ u$  is a gradient, one may rewrite the above equation as
\begin{equation}\label{A1}
 (\cP u)_t+ \cP((u+V)\cdot \nabla \cP u) - \Delta \cP u = -\cP(a V_t + a u_t)-\cP R_1
 \end{equation}
 with 
\begin{multline}\label{A2}
 R_1 :=  (1+a)\cP u \cdot \nabla(V+\cQ u)  + (1+a)V\cdot \nabla \cQ u + (1+a) \cQ u \cdot \nabla V 
 \\+ a (u+V) \cdot \nabla\cP u +aV\cdot\nabla V +a\cQ u\cdot\nabla\cQ u.\end{multline}
Let us start the computations. The general approach is adapted from \cite{D1}: we localize equations 
\eqref{B1}, \eqref{a} and \eqref{A1} in the frequency space by means of the dyadic operators 
$\ddj,$ and perform suitable energy estimates. The key point 
is that we strive for  time pointwise estimates of $u,$  $a$ and $\nu\nabla a$ 
in the same (Besov) space. 
\smallbreak
In what follows, we denote by 
 $a^\ell$ and $a^h$  the low and high frequencies  parts  of $a,$ respectively (see \eqref{eq:lhf})
and set
$$
\begin{array}{ll}
 X_d(T):= \|\cQ u, a, \nu\nabla a\|_{L_\infty(0,T;\dot B^{d/2-1}_{2,1})}, &\!\!\!\quad 
 Y_d(T):= \| \cQ u_t, \nu \nabla^2 \cQ u,\nu\nabla^2a^\ell,\nabla a^h\|_{L_1(0,T;\dot B^{d/2-1}_{2,1})}, \\[5pt]
 Z_d(T):= \|\cP u\|_{L_\infty(0,T;\dot B^{d/2-1}_{2,1})}, &\!\!\!\quad 
 W_d(T):= \|\cP u_t, \nabla^2\cP u\|_{L_1(0,T;\dot B^{d/2-1}_{2,1})}.
 \end{array}
$$
We assume   that the maximal solution $(\varrho=1+a,v)$ of \eqref{eq:NSC} corresponding to data $(\varrho_0,v_0)$ 
is defined on the time interval $[0,T_\star)$ and satisfies \eqref{eq:reg}, 
and we fix some $M\geq0$ so that the `incompressible solution' $V$  to \eqref{NS} fulfills
\begin{equation}\label{f3a}
V_d(T):=  \|V\|_{L_\infty(0,T;\dot B^{d/2-1}_{2,1})}+ \|V_t,\nabla^2 V\|_{L_1(0,T;\dot B^{d/2-1}_{2,1})}
\leq M\quad\hbox{for all }\ T\geq0.
\end{equation}

As already pointed out  and proved in Appendix, in the 2D case, number $M$  may 
be expressed in terms of  $\|V_0\|_{\dot B^0_{2,1}(\R^2)}.$ 
\medbreak
We claim that if $\nu$ is large enough then one may find some (large) $D$ and (small) $\delta$
so that for all $T<T_\star,$  the following bounds are  valid:
\begin{equation}\label{f3}
 X_d(T)+Y_d(T) \leq D\quad\hbox{and}\quad
 Z_d(T)+W_d(T) \leq \delta.
\end{equation}


\subsubsection*{Step 1. Estimates for the divergence-free part of the velocity}

Applying $\ddj$ to \eqref{A1},   taking the $L_2$ inner product with $\ddj\cP u$ then using that $\cP^2=\cP,$ we discover that
$$\displaylines{
\frac 12\frac d{dt}\|\ddj\cP u\|_{L_2}^2+\|\nabla\ddj\cP u\|_{L_2}^2+\int\bigl(\ddj(u+V)\cdot\nabla\cP u\bigr)\cdot\ddj\cP u\,dx
\hfill\cr\hfill=-\int \ddj(a V_t+au_t+R_1)\cdot\ddj\cP u\,dx.}
$$
Then  using Bernstein's inequalities in the second term, 
swapping operators $\ddj$  and $u+V$ in the third term, and integrating by parts, we get
 we get for some universal constant $c>0,$
\begin{multline}\label{eq:Pu}
\frac 12\frac d{dt}\|\ddj\cP u\|_{L_2}^2+c\|\nabla^2\ddj\cP u\|_{L_2}\|\ddj\cP u\|_{L_2}\leq
\frac12\int|\ddj\cP u|^2\,\div u\,dx \\+ \int\bigl([u+V,\ddj]\cdot\nabla\cP u\bigr)\cdot\ddj\cP u\,dx
-\int \ddj(a V_t+au_t+R_1)\cdot\ddj\cP u\,dx.
\end{multline}
It is well known (see e.g. Lemma 2.100 in \cite{BCD}) that the commutator term may be estimated as follows:
\begin{equation}\label{eq:com}
2^{j({d/2-1})}\bigl\|\bigl[u+V,\ddj]\cdot\nabla\cP u\bigr\|_{L_2}\leq C c_j\|\nabla(u+V)\|_{\dot B^{d/2}_{2,1}}\|\cP u\|_{\dot B^{d/2-1}_{2,1}}
\quad\hbox{with}\quad \sum_{j\in\Z} c_j=1.
\end{equation}
Hence dividing  (formally) \eqref{eq:Pu} by $\|\dot \Delta_j \cP u\|_{L_2}$, multiplying by $2^{j({d/2-1})}$, integrating \eqref{eq:Pu}, remembering that $\cP u|_{t=0}=0$  and summing over $j$, we obtain
\begin{multline}\label{D1}
 \|\cP u\|_{L_\infty(0,T;\dot B^{d/2-1}_{2,1})} +  \|\nabla^2 \cP u\|_{L_1(0,T;\dot B^{d/2-1}_{2,1})} \\\lesssim
 \int_0^T\|\nabla(u+V)\|_{\dot B^{d/2}_{2,1}}\|\cP u\|_{\dot B^{d/2-1}_{2,1}}\,dt
   + \int_0^T\|a V_t + a u_t\|_{\dot B^{d/2-1}_{2,1}}\,dt
   +\int_0^T\|R_1\|_{\dot B^{d/2-1}_{2,1}}\,dt.
\end{multline}
Next we see, thanks to Lemma \ref{l:para1} that
$$
\begin{aligned}
 \|a V_t + a u_t\|_{L_1(0,T;\dot B^{d/2-1}_{2,1})}&\lesssim
\nu^{-1} \|\cQ u_t, \cP u_t, V_t\|_{L_1(0,T;\dot B^{d/2-1}_{2,1})} \| \nu a\|_{L_\infty(0,T;\dot B^{d/2}_{2,1})}\\&\lesssim 
\nu^{-1} (Y_d(T)+W_d(T) + V_d(T)) X_d(T).
\end{aligned}
$$
In order to bound  $R_1,$  we use the fact that
$$
 \| (1+a)\cP u \cdot \nabla(V+\cQ u)\|_{\dot B^{{d/2-1}}_{2,1}} 
\lesssim  (1+\|a\|_{\dot B^{d/2}_{2,1}})  \|\nabla(V+\cQ u)\|_{\dot B^{d/2}_{2,1}} \|\cP u\|_{\dot B^{d/2-1}_{2,1}}
 $$
and
$$
\|(1+a)(\cQ u \cdot \nabla V+V\cdot\nabla\cQ u)\|_{\dot B^{{d/2-1}}_{2,1}} \lesssim
  (1+\|a\|_{\dot B^{d/2}_{2,1}})\|\cQ u\|_{\dot B^{d/2}_{2,1}} \|V\|_{\dot B^{d/2}_{2,1}},
$$
whence, integrating on $[0,T]$ and using the interpolation inequality
$$
\|z\|_{\dot B^{d/2}_{2,1}}\leq C \|z\|_{\dot B^{d/2-1}_{2,1}}^{1/2}\|\nabla^2z\|_{\dot B^{d/2-1}_{2,1}}^{1/2}
\quad\hbox{for }\ z=V,\,\cQ u,$$
we get
$$
\|(1+a)(\cQ u \cdot \nabla V+V\cdot\nabla\cQ u)\|_{L_1(0,T;\dot B^{{d/2-1}}_{2,1})} \lesssim
 (1+\nu^{-1}X_d(T)) \nu^{-1/2}X_d(T)^{1/2}Y_d^{1/2}(T)V_d(T).$$
Finally, we have 
$$\begin{aligned}
\|a(u+V)\cdot\nabla\cP u\|_{L_1(0,T;\dot B^{{d/2-1}}_{2,1})} &\lesssim
\|a\|_{L_\infty(0,T;\dot B^{d/2}_{2,1})}\|u+V\|_{L_\infty(0,T;\dot B^{d/2-1}_{2,1})}
\|\nabla\cP u\|_{L_1(0,T;\dot B^{{d/2}}_{2,1})}\\&\lesssim  \nu^{-1}X_d(T)    \bigl(Z_d(T)+X_d(T)+V_d(T)\bigr) W_d(T),
\end{aligned}
$$
and
$$\displaylines{
\|aV\cdot\nabla V+a\cQ u\cdot\nabla\cQ u\|_{L_1(0,T;\dot B^{d/2-1}_{2,1})}\lesssim \|a\|_{L_\infty(0,T;\dot B^{d/2}_{2,1})}
\bigl(\|V\|_{L_\infty(0,T;\dot B^{d/2-1}_{2,1})}\|\nabla V\|_{L_1(0,T;\dot B^{d/2}_{2,1})}\hfill\cr\hfill
+\|\cQ u\|_{L_\infty(0,T;\dot B^{d/2-1}_{2,1})}\|\nabla\cQ u\|_{L_1(0,T;\dot B^{d/2}_{2,1})}\bigr).}
$$
Therefore,  we obtain 
$$\displaylines{
 Z_d(T)+W_d(T) \lesssim   \int_0^T \|\nabla V,\nabla\cP u,\nabla\cQ u\|_{\dot B^{d/2}_{2,1}} Z_d(t)\, dt
\hfill\cr\hfill+ \nu^{-1} (Y_d(T)+W_d(T) + V_d(T)) X_d(T)
  + \nu^{-1/2}X_d(T)^{1/2}Y_d^{1/2}(T)V_d(T) (1+\nu^{-1}X_d(T))\hfill\cr\hfill
+  \nu^{-1}X_d(T)    \bigl(Z_d(T)+X_d(T)+V_d(T)\bigr) W_d(T)
+\nu^{-1}X_d(T) (V_d^2(T)+\nu^{-1}X_d(T)Y_d(T)).}$$
 Assuming from now on that
\begin{equation}\label{eq:smalla}
X_d(T)\ll\nu,
\end{equation}
and using  Gronwall lemma, we conclude that 
\begin{multline}\label{eq:step1final}
 Z_d(T)+W_d(T) \leq Ce^{C\int_0^t \|\nabla V,\nabla\cP u,\nabla \cQ u\|_{\dot B^{d/2}_{2,1}}\,d\tau} 
\Bigl( \nu^{-1} (Y_d(T) + V_d(T)+W_d(T)) X_d(T)\\
  + \nu^{-1/2}X_d(T)^{1/2}Y_d^{1/2}(T)V_d(T)\\
+  \nu^{-1}X_d(T)    \bigl(Z_d(T)+X_d(T)+V_d(T)\bigr) W_d(T)
+\nu^{-1}X_d(T) V_d^2(T)\Bigr)\cdotp\end{multline}


\subsubsection*{Step 2.  Estimate on the potential part of the velocity and on the density}
To estimate the potential part of the velocity, we are required to consider the momentum and continuity equations
altogether. Now, localizing \eqref{a} and \eqref{B1} according to Littlewood-Paley operators, we discover that  
\begin{align}\label{p1a}
  &a_{j,t} +(u+V)\cdot\nabla a_j  +\div \cQ u_j =g_j\\\label{p1b}
 &\cQ u_{j,t} +  \cQ( (u+V) \cdot \nabla \cQ u_j )- \nu \Delta \cQ u_j + \nabla a_j = f_j
 \end{align}
where
\begin{equation}\label{p2}
 a_j := \ddj a,\qquad \cQ u_j := \ddj \cQ u,
\end{equation}
\begin{equation}\label{p3}
 g_j:= -\ddj(a\div \cQ u) - [\ddj,(u+V)]\cdot\nabla a
\end{equation}
and
\begin{equation}\label{p4}
 f_j:=-\ddj\cQ ( a V_t +a u_t)- \ddj\cQ R_2 - [\ddj, u+V]\cdot\nabla \cQ u.
\end{equation}

We follow an  energy method to bound each term $(a_j,\cQ u_j).$
More precisely, testing \eqref{p1a} and \eqref{p1b} by $a_j$ and $\cQ u_j,$ respectively, yields
\begin{equation}\label{p5}
 \frac{1}{2} \frac{d}{dt} \int a_j^2\, dx + \int a_j\div \cQ u_j\, dx = 
 \frac{1}{2} \int \div u \,  a_j^2 \,dx + \int g_j a_j\, dx
\end{equation}
and 
\begin{multline}\label{p6}
\frac 12 \frac{d}{dt} \int |\cQ u_j|^2 \,dx + \nu \int |\nabla \cQ u_j|^2 \,dx 
- \int a_j \,\div\cQ u_j \,dx\\= \frac12 \int \div u |\cQ u_j|^2\,dx + \int f_j\cdot \cQ u_j\, dx.
\end{multline}

We next want an estimate for $\|\nabla a_j\|_{L_2}^2.$ {}From \eqref{p1a}, we have
\begin{equation}\label{p7}
 \nabla a_{j,t} +  (u+V)\cdot \nabla \nabla a_j + \nabla \div \cQ u_j = 
 \nabla g_j-  \nabla (u+V) \cdot\nabla a_j.
\end{equation}
Testing that equation  by $\nabla a_j$ yields
\begin{multline}\label{p8a}
\frac12 \frac{d}{dt}\int |\nabla a_j|^2 \, dx + \int\bigl((u+V)\cdot\nabla\nabla a_j)\cdot \nabla a_j\,dx
+\int\nabla\div\cQ u_j\cdot\nabla a_j\,dx\\= \int\bigl(\nabla g_j-  \nabla (u+V) \cdot\nabla a_j)\cdot\nabla a_j\,dx.
\end{multline}

To eliminate the highest order term, namely the one with $\nabla\div\cQ u_j,$ it is suitable 
to combine the above equality with a relation involving $\int\cQ u_j\cdot \nabla a_j\,dx.$ Now, 
testing  \eqref{p7}  by $\cQ u_j$ and the  momentum equation by $\nabla a_j,$ we get
\begin{multline}\label{p8}
 \frac{d}{dt}\int\cQ u_j \cdot\nabla a_j\, dx
+\int (u+V)\cdot\nabla(\cQ u_j\cdot\nabla a_j)\,dx 
 -\nu\int \Delta\cQ u_j \cdot\nabla a_{j}\,dx\\ + \int |\nabla a_j|^2 \,dx +\int\nabla\div\cQ u_j\cdot\cQ u_j\,dx =
 \int \bigl( \nabla g_j -\nabla(u+V)\cdot\nabla a_j\bigr)\cdot\cQ u_j\,dx  +\int  f_j\cdot \nabla a_j \,dx.
\end{multline}
Note that by integration by parts, we have
$$
\int (u+V)\cdot\nabla(\cQ u_j\cdot\nabla a_j)\,dx = -\int \cQ u_j\cdot\nabla a_j\:\div u\,dx.
$$

Hence adding  $\nu$ times  \eqref{p8a} to \eqref{p8} and noting that  $\Delta \cQ u_j \equiv \nabla \div \cQ u_j,$
the highest order terms cancel out, and we get 
$$\displaylines{
\frac12 \frac{d}{dt}\int\bigl(\nu|\nabla a_j|^2+2\cQ u_j \cdot\nabla a_j\bigr)\,dx
+ \nu\int(|\nabla a_j|^2-|\Delta\cQ u_j|^2)\,dx\hfill\cr\hfill=
\int \biggl(\frac\nu2|\nabla a_j|^2+\cQ u_j\cdot\nabla a_j\biggr)\div u\,dx
+\nu\int\bigl(\nabla g_j-  \nabla (u+V) \cdot\nabla a_j)\cdot\nabla a_j\,dx\hfill\cr\hfill
+ \int \bigl( \nabla g_j -\nabla(u+V)\cdot\nabla a_j\bigr)\cdot\cQ u_j\,dx  +\int  f_j\cdot \nabla a_j \,dx.}
$$
After multiplying the above equality by $\nu$ and adding up twice \eqref{p5} and \eqref{p6}, we get
\begin{multline}\label{p9}
\frac12 \frac{d}{dt}\cL_j^2 +  \nu\int \Bigl(|\nabla \cQ u_j|^2 + |\nabla a_j|^2\Bigr) dx \\=
 \int \bigl(2g_ja_j+2f_j\cdot\cQ u_j+\nu^2\nabla g_j\cdot\nabla a_j+\nu\nabla g_j\cdot\cQ u_j+\nu f_j\cdot\nabla a_j\bigr)dx\\
 +\frac12\int\cL_j^2\,\div u\,dx-\nu\int \bigl(\nabla(u+V)\cdot\nabla a_j\bigr)\cdot(\nu\nabla a_j+\cQ u_j)\,dx
 \end{multline}
with 
\begin{equation}\label{p10}
 \cL^2_j := \int  \bigl(2a_j^2 + 2|\cQ u_j|^2 + 2\nu \cQ u_j \cdot\nabla a_j + |\nu\nabla a_j|^2)\,dx.
\end{equation}

At this stage, two fundamental  observations are in order. 
First, we obviously have
\begin{equation}\label{p11}
\cL_j\approx\|(\cQ u_j,a_j,\nu\nabla a_j)\|_{L_2}\quad\hbox{ for all }\ j\in\Z
\end{equation}
and, second, 
\begin{equation}\label{p11b}
 \nu\int ( |\nabla \cQ u_j|^2 + |\nabla a_j|^2)\, dx \geq c\min(\nu2^{2j},\nu^{-1})\cL_j^2.
\end{equation}
Therefore  \eqref{p9}, \eqref{p11} and \eqref{p11b}  lead to 
$$\frac12\frac{d}{dt} \cL_j ^2+ c\min(\nu2^{2j},\nu^{-1}) \cL_j^2\leq 
\biggl(\frac12\|\div u\|_{L_\infty}+C\|\nabla(u+V)\|_{L_\infty}\biggr)\cL_j^2
+ C\|[g_j,f_j,\nu\nabla g_j]\|_{L_2}  \cL_j,$$
whence integrating in time, 
\begin{multline}\label{p13}
 \cL_j(t)+ c\min(\nu2^{2j},\nu^{-1}) \int_0^t\cL_j\,d\tau\\\leq \cL_j(0)
+C\int_0^t\|\nabla(u+V)\|_{L_\infty}\cL_j\,d\tau+ C\int_0^t\|[g_j,f_j,\nu\nabla g_j]\|_{L_2}\,d\tau.\end{multline}
Note that we lost the expected parabolic smoothing of $\cQ u$
because $\min(\nu2^{2j},\nu^{-1})=\nu^{-1}$ for large  $j$'s.  
However, it  may be recovered  by  starting  directly from \eqref{p6} and  integrating by parts in the term with $a_j\div\cQ u_j.$
After using Bernstein and H\"older inequalities, we arrive at
$$
\frac 12\frac d{dt}\|\cQ u_j\|_{L_2}^2+c\nu2^{2j}\|\cQ u_j\|_{L_2}^2\leq  \|\nabla a_j\|_{L_2}\|\cQ u_j\|_{L_2}
+\frac12\|\div u\|_{L_\infty}\|\cQ u_j\|_{L_2}^2+\|f_j\|_{L_2}\|\cQ u_j\|_{L_2},
$$
whence, integrating in time,
$$\displaylines{
\|\cQ u_j(t)\|_{L_2}+c\nu2^{2j}\int_0^t\|\cQ u_j\|_{L_2}\,d\tau\leq \|\cQ u_j(0)\|_{L_2}\hfill\cr\hfill+
\int_0^t  \|\nabla a_j\|_{L_2}\,d\tau+\frac12\int_0^t\|\div u\|_{L_\infty}\|\cQ u_j\|_{L_2}\,d\tau
+\int_0^t\|f_j\|_{L_2}\,d\tau.}
$$
Putting together with \eqref{p13}, remembering  \eqref{p11}, multiplying by $2^{j({d/2-1})}$ and eventually summing up over $j\in\Z,$ 
we end up with the following fundamental inequality:
\begin{multline}\label{p14}
\|(a,\nu\nabla a,u)(t)\|_{\dot B^{d/2-1}_{2,1}}+\nu\int_0^t\|\nabla a^\ell,\nabla\cQ u\|_{\dot B^{d/2}_{2,1}}\,d\tau
+\int_0^t\|a^h\|_{\dot B^{d/2}_{2,1}}\,d\tau\\ \lesssim
\|(a,\nu\nabla a,u)(0)\|_{\dot B^{d/2-1}_{2,1}}+ \int_0^t\|\nabla(u+V)\|_{L_\infty}\|(a,\nu\nabla a,\cQ u)\|_{\dot B^{d/2-1}_{2,1}}\,d\tau\\+\int_0^t \sum_{j\in\Z}2^{j(d/2-1)}  \|[g_j,f_j,\nu\nabla g_j]\|_{L_2}\, d\tau,
\end{multline}
where notations  $a^\ell$ and $a^h$  have been defined in \eqref{eq:lhf}.
\medbreak
To complete the proof of estimates for $a$ and $\cQ u,$  we now have to get suitable bounds for the last term
in \eqref{p14}.  
Let us start with the study of $g_j$ defined by \eqref{p3}. First we see that, 
by virtue of Lemma \ref{l:para1}, we have 
\begin{equation}\label{p16}
 \|a\,\div \cQ u\|_{\dot B^{d/2-1}_{2,1}}\lesssim \|\div \cQ u \|_{\dot B^{d/2}_{2,1}}\|a\|_{\dot B^{d/2-1}_{2,1}} 
\end{equation}
and, using rule and, again, Lemma \ref{l:para1},
\begin{equation}\label{p17}
 \nu\|\nabla(a\div\cQ u)\|_{\dot B^{d/2-1}_{2,1}} \lesssim 
 \|\nabla\div \cQ u \|_{\dot B^{d/2-1}_{2,1}}\|\nu a\|_{\dot B^{d/2}_{2,1}}  + 
 \|\div \cQ u \|_{\dot B^{d/2}_{2,1}} \|\nu \nabla a\|_{\dot B^{d/2-1}_{2,1}}.
\end{equation}
The commutator term may be bounded as follows (the first inequality stems from  Lemma 2.100 in \cite{BCD}, 
and the second one may be deduced from that lemma and Leibniz rule):
\begin{align}\label{p18}
\sum_{j\in\Z}2^{j({d/2-1})} \| [\ddj,(u+V)]\nabla a \|_{L_2}&\leq C\|\nabla(u+V) \|_{\dot B^{d/2}_{2,1}}\|a\|_{\dot B^{d/2-1}_{2,1}},\\\label{p19}
\sum_{j\in\Z} 2^{j({d/2-1})}\nu\|\nabla ([\ddj,(u+V)]\nabla a ) \|_{L_2}&\leq C  \|\nabla(u+V)\|_{\dot B^{d/2}_{2,1}}
 \|\nu\nabla a\|_{\dot B^{d/2-1}_{2,1}}.
\end{align}
Hence, putting \eqref{p16} to \eqref{p19} together, we get 
\begin{equation}\label{p19b}
\sum_{j\in\Z} 2^{j(d/2-1)}\|g_j,\nu\nabla g_j\|_{L_2}\leq C  \|\nabla(u+V)\|_{\dot B^{d/2}_{2,1}}
\bigl(\|a\|_{\dot B^{d/2-1}_{2,1}}+\nu\|a\|_{\dot B^{d/2}_{2,1}}\bigr).
\end{equation}

Next,  let us bound  $f_j$ defined in \eqref{p4}. To handle  the terms corresponding to $R_2$ (see 
\eqref{B2}), we use the fact that 
\begin{equation}\label{p19a}
 \|(1\!+\!a)(u\!+\!V)\cdot \nabla (\cP u+V)\|_{\dot B^{d/2-1}_{2,1}}\\\lesssim\bigl(1+\|a\|_{\dot B^{d/2}_{2,1}}\bigr)
 \|(u,V)\|_{\dot B^{d/2-1}_{2,1}} \|(\nabla \cP u,\nabla V)\|_{\dot B^{d/2}_{2,1}},\end{equation}
\begin{equation}\label{p21}
\|a(u+V)\cdot\nabla\cQ u\|_{\dot B^{d/2-1}_{2,1}}\lesssim \|a\|_{\dot B^{d/2}_{2,1}} \|(u,V)\|_{\dot B^{d/2-1}_{2,1}}
\|\nabla\cQ u\|_{\dot B^{d/2-1}_{2,1}},
\end{equation}
and (see (\ref{B1}))
\begin{equation}\label{p21a}
\|k(a)\nabla a\|_{\dot B^{d/2-1}_{2,1}}\lesssim \|a\|_{\dot B^{d/2}_{2,1}}^2.
\end{equation}
As for $g,$ the commutator term of $f$ may be bounded according to Lemma 2.100 in \cite{BCD}:
\begin{equation}
\sum_{j\in\Z}2^{j({d/2-1})}\|[\ddj, u+V]\nabla \cQ u\|_{L_2} \leq 
 C\|\nabla(u+V) \|_{\dot B^{d/2}_{2,1}}\|\cQ u\|_{\dot B^{d/2-1}_{2,1}}.
\end{equation}
Finally, for the   terms with the time derivative, we have 
\begin{equation}\label{p22}
 \|aV_t\|_{\dot B^{d/2-1}_{2,1}} +\|a u_t\|_{\dot B^{d/2-1}_{2,1}} \leq C
 \|V_t,u_t\|_{\dot B^{d/2-1}_{2,1}} \|a\|_{\dot B^{d/2}_{2,1}}.
\end{equation}
From \eqref{p19a} to \eqref{p22}, we conclude that
\begin{multline}\label{fj}
\sum_{j\in\Z}2^{j(d/2-1)}\|f_j\|_{L_2}\leq C\Bigl(\|(u,V)\|_{\dot B^{d/2-1}_{2,1}}\|\nabla\cP u,\nabla V\|_{\dot B^{d/2}_{2,1}}\\
+\|a\|_{\dot B^{d/2}_{2,1}}\bigl(\|(u,V)\|_{\dot B^{d/2-1}_{2,1}}\|\nabla u,\nabla V\|_{\dot B^{d/2}_{2,1}}
+  \|V_t,u_t\|_{\dot B^{d/2-1}_{2,1}}+\|a\|_{\dot B^{d/2}_{2,1}}\bigr)\Bigr)\cdotp
\end{multline}
Putting \eqref{p19b} and \eqref{fj}  together in \eqref{p14}  gives us for all $0\leq T<T_\star,$
\begin{multline}\label{f1}
 \| \cQ u, a, \nu\nabla a\|_{L_\infty(0,T;\dot B^{d/2-1}_{2,1})} +
 \| \cQ u_t, \nu \nabla^2 \cQ u,\nu\nabla^2 a^\ell,
 \nabla a^h\|_{L_1(0,T;\dot B^{d/2-1}_{2,1})}\\ \lesssim   \|(\cQ u, a, \nu\nabla a)(0)\|_{\dot B^{d/2-1}_{2,1}}
 +  \int_0^T\|\nabla u,\nabla V \|_{\dot B^{d/2}_{2,1}}\|a, \nu\nabla a,\cQ u\|_{\dot B^{d/2-1}_{2,1}}\,d\tau\\
  +\bigl(1+\|a\|_{L_\infty(0,T;\dot B^{d/2}_{2,1})}\bigr)\int_0^T
 \|(\cP u,\cQ u,V)\|_{\dot B^{d/2-1}_{2,1}} \|(\nabla \cP u,\nabla V)\|_{\dot B^{d/2}_{2,1}}\,d\tau  \\
  +\|a\|_{L_\infty(0,T;\dot B^{d/2}_{2,1})}  \int_0^T\|(\cP u,\cQ u,V)\|_{\dot B^{d/2-1}_{2,1}} \|\nabla \cQ u\|_{\dot B^{d/2}_{2,1}}\,d\tau\\
 + \|V_t,u_t\|_{L_1(0,T;\dot B^{d/2-1}_{2,1})} \|a\|_{L_\infty(0,T;\dot B^{d/2}_{2,1})}  +  \|a\|_{L_2(0,T;\dot B^{d/2}_{2,1})}^2.
\end{multline}
Hence,    using obvious interpolation to bound the last term, and also  the fact that \eqref{eq:smalla} implies that 
\begin{equation}\label{eq:smallabis}
\|a\|_{L_\infty(0,T;\dot B^{d/2}_{2,1})}\ll1,
\end{equation}
we get
\begin{multline}\label{f2}
 \| \cQ u, a, \nu\nabla a\|_{L_\infty(0,T;\dot B^{d/2-1}_{2,1})} +
 \| \cQ u_t, \nu \nabla^2 \cQ u,\nu\nabla^2 a^\ell,
 \nabla a^h\|_{L_1(0,T;\dot B^{d/2-1}_{2,1})}\\ \lesssim   \|(\cQ u, a, \nu\nabla a)(0)\|_{\dot B^{d/2-1}_{2,1}}
 +  \int_0^t\|\nabla\cP u,\nabla\cQ u,\nabla V \|_{\dot B^{d/2}_{2,1}}\|a, \nu\nabla a,\cQ u\|_{\dot B^{d/2-1}_{2,1}}\,d\tau\\
  +\|(\cP u,V)\|_{L_\infty(0,T;\dot B^{d/2-1}_{2,1})} \|(\nabla \cP u,\nabla V)\|_{L_1(0,T;\dot B^{d/2}_{2,1})} \\ 
  +\|a\|_{L_\infty(0,T;\dot B^{d/2}_{2,1})}  \|(\cP u,V)\|_{L_\infty(0,T;\dot B^{d/2-1}_{2,1})} \|\nabla \cQ u\|_{L_1(0,T;\dot B^{d/2}_{2,1})}\\
 + \|V_t,\cP u_t\|_{L_1(0,T;\dot B^{d/2-1}_{2,1})} \|a\|_{L_\infty(0,T;\dot B^{d/2}_{2,1})}  
 +  \nu^{-1}\bigl(\|a^\ell\|_{L_\infty(0,T;\dot B^{d/2-1}_{2,1})} \|\nu\nabla a^\ell\|_{L_1(0,T;\dot B^{d/2}_{2,1})}\\
  +  \|\nu a^h\|_{L_\infty(0,T;\dot B^{d/2}_{2,1})} \|a^h\|_{L_1(0,T;\dot B^{d/2}_{2,1})}\bigr).
\end{multline}
Hence, from  Gronwall lemma,
\begin{multline}\label{f4}
 X_d(T)+Y_d(T)\leq Ce^{C\int_0^t  \|\nabla\cP u,\nabla\cQ u,\nabla V \|_{\dot B^{d/2}_{2,1}}\,d\tau}\biggl(
 X_d(0)+(V_d(T)+Z_d(T))(V_d(T)+W_d(T))\\+\nu^{-2}X_d(T)Y_d(T)(V_d(T)+Z_d(T))
 +\nu^{-1}(V_d(T)+Y_d(T)+W_d(T))X_d(T)\biggr)\cdotp \end{multline}
 

\subsubsection*{Step 3. Global-in-time closure of the estimates}

 Assuming that 
 \begin{equation}\label{eq:smallD1}
 \nu^{-1}D \ll1,
 \end{equation}
 Inequality \eqref{f4} and hypotheses \eqref{f3a} and \eqref{f3} imply that
$$\displaylines{
 X_d(T)+Y_d(T)\leq Ce^{C(M+\nu^{-1}D+\delta)} \Bigl(X_d(0)+(M+\delta)^2
\hfill\cr\hfill+\nu^{-2}D(M+\delta)X_d(T)+\nu^{-1}(D+\delta+M)X_d(T)\Bigr)}
$$
while \eqref{eq:step1final} yields
$$\displaylines{
 Z_d(T)+W_d(T) \leq CDe^{C(M+\nu^{-1}D+\delta)} \bigl( \nu^{-1}(M+\delta+D)
 \hfill\cr\hfill+\nu^{-1/2} M +\nu^{-1}(M+D+\delta)W_d(T)+\nu^{-1}M^2\bigr) \cdotp}
$$
Hence, assuming in addition that
 \begin{equation}\label{eq:smallD2}
\nu^{-1}D\leq M\quad\hbox{and}\quad \delta\leq \max\{M,1\},
 \end{equation}
we get (enlarging $C$ as the case may be)
$$\displaylines{
 X_d(T)+Y_d(T)\leq Ce^{CM} \Bigl(X_d(0)+M^2+1
+\nu^{-1}(M+D)X_d(T)\Bigr),\cr
 Z_d(T)+W_d(T) \leq CDe^{CM}\Bigl(\nu^{-1}D+\nu^{-1/2}M+\nu^{-1}M^2+\nu^{-1}(M+D)W_d\Bigr)\cdotp}
 $$
Therefore, if we make the assumption that
 \begin{equation}\label{eq:smallD2b}
 D(D+M+1)e^{CM} \ll\nu,
  \end{equation}
then we end up with 
\begin{equation}\label{f5b}
 X_d(T)+Y_d(T)\leq Ce^{CM} \bigl(X_d(0) +M^2+1\bigr)
\end{equation}
and
\begin{equation}\label{f8}
 Z_d(T)+W_d(T) \leq  CDe^{CM} \bigl(\nu^{-1/2} M+\nu^{-1}(D+M^2+1)\bigr) \cdotp
\end{equation}
So it is natural to take first
\begin{equation}\label{f9}
D:= Ce^{CM}(X_d(0)+M^2+1)
\end{equation}
 and then to set 
\begin{equation}\label{f10}
\delta:=  Ce^{CM}\bigl(X_d(0)+M^2\bigr) \Bigl( \nu^{-1/2}M +\nu^{-1}(X_d(0)+ M^2+1\bigr)\Bigr)\cdotp \end{equation}
Now, assuming that for a suitably large (universal) constant $C$ we have
\begin{equation}\label{eq:nu}
Ce^{CM}(X_d(0)+1+M^2)\leq\sqrt{\nu},
\end{equation} 
we see that Conditions   \eqref{eq:smallD1} and  \eqref{eq:smallD2b}  are fulfilled
(and thus also   \eqref{eq:smallD2} as it is weaker). 
\medbreak
Let us recap:  if $\nu$ and the compressible part of the data 
fulfill  \eqref{eq:nu}  then defining $D$ and $\delta$ according to \eqref{f9} and \eqref{f10}
ensures that \eqref{f3} is fulfilled for all $T<T_\star.$ Then,  combining with the 
continuation criterion recalled in \eqref{eq:ex}, one can conclude that $T_\star=+\infty$
and that \eqref{f3} is satisfied for all time. This completes the proof of Theorems \ref{th:2D} and \ref{th:dD}. 
Finally,   in the 2D case, Theorem \ref{Th:NS} enables us to  take
\begin{equation}\label{eq:M}
M= C\|\cP v_0\|_{\dot B^0_{2,1}}\exp\Bigl(\frac{C}{\mu^4}\|\cP v_0\|^4_{L_2}\Bigr),
\end{equation}
which provides us with an explicit  lower bound for $\nu$ depending only  on the initial data, through \eqref{eq:nu}.


\section{Appendix}

We here consider the global well-posedness issue of the incompressible two-dimensional Navier-Stokes system
in the critical Besov spaces setting. 
 Although essentially the same result has been proved  in \cite{D3} (see Theorem 6.3 therein), 
 we here provide another (different and more elementary) proof for the reader convenience. 
\begin{thm}\label{Th:NS}
Let $V_0$ be in $\dot B^{0}_{2,1}(\R^2)$ with  $\div V_0=0$. Then there exists a unique solution to \eqref{NS}
such that
\begin{equation}\label{V-ene1}
V \in L_\infty(\R_+;\dot B^0_{2,1}(\R^2))\quad\hbox{and}\quad V_t, \nabla^2 V\in L_1(\R_+;\dot B^0_{2,1}(\R^2)).
\end{equation}
Furthermore, the following inequality is satisfied for all $T\geq0$:
\begin{equation}\label{V10}
 \|V\|_{L_\infty(0,T;\dot B^0_{2,1})} + \|V_t,\nabla^2V\|_{L_1(0,T;\dot B^{0}_{2,1})} 
 \leq C\|V_0\|_{\dot B^0_{2,1}}\exp\Bigl(\frac{C}{\mu^4}\|V_0\|^4_{L_2}\Bigr)
\end{equation}
for some universal constant $C.$
\end{thm}
\begin{proof} 
First recall that the space $\dot B^0_{2,1}(\R^2)$ for initial 
velocity embeds in $L_2(\R^2).$ Hence we have  $V_0\in L_2(\R^2)$  and the pioneering works by J. Leray in \cite{Leray}
delivers us global in time weak solutions satisfying the energy estimate
\begin{equation}\label{V1}
 \sup_{t\in\R_+} \|V(t)\|_{L_2(\R^2)}^2 + 2\mu\int_0^{\infty} \|\nabla V\|_{L_2(\R^2)}^2\, dx = \|V_0\|_{L_2(\R^2)}^2.
\end{equation}
Then  the classical  result of Olga Alexandrovna  \cite{OL} provides 
a unique  global in time  solution  $V\in L_\infty(\R_+;L_2(\R^2))\cap L_2(\R_+;\dot H^1(\R^2))$. Here we want to improve the regularity to 
the class defined by (\ref{V-ene1}).

Now, real  interpolation applied to (\ref{V1}) gives 
\begin{equation*}
V\in \left( L_\infty(\R_+;L_2(\R^2)), L_2(\R_+;\dot H^1(\R^2))\right)_{1/2,1},
\end{equation*}
which implies that
$V \in L_4(\R_+;\dot B^{1/2}_{2,1}(\R^2)).$
\medbreak
To take advantage of that  information, we look at the equation satisfied by $V$
 as a nonlinear modification of the Stokes system, namely
\begin{equation}\label{V-Stokes}
 \begin{array}{l}
V_t - \mu\Delta V + \nabla\Pi = -V \cdot \nabla V, \\
\div V=0, \\
V|_{t=0}=V_0.
 \end{array}
\end{equation}

Using the endpoint maximal regularity estimates of  the Stokes system in homogeneous Besov 
spaces (which, in $\R^2,$  coincide with those  for the heat equation), we may write 
\begin{equation}\label{V5}
 \|V\|_{L_\infty(0,T;\dot B^0_{2,1})} + \|V_t,\mu\nabla^2V\|_{L_1(0,T;\dot B^0_{2,1})} 
\leq C(\|V\cdot \nabla V\|_{L_1(0,T;\dot B^0_{2,1})} + \|V_0\|_{\dot B^0_{2,1}}).
\end{equation}

Next we have to bound  $\|V\cdot \nabla V\|_{L_1(0,T;\dot B^0_{2,1})}$. From the energy balance \eqref{V1}
and the interpolation property pointed out above, we know that
\begin{equation}\label{eq:interpo}
\mu^{1/4}\|V\|_{L_4(\R_+;\dot B^{1/2}_{2,1})} \leq C \|V_0\|_{L_2}.\end{equation}
Furthermore, product laws in Besov spaces (Lemma \ref{l:para1}) ensure that  
\begin{equation}\label{V7}
 \|V \cdot \nabla V \|_{\dot B^0_{2,1}} \leq C \|V\|_{\dot B^{1/2}_{2,1}} 
\|\nabla V \|_{\dot B^{1/2}_{2,1}}.
\end{equation}
Hence integrating \eqref{V7} on the time interval $[0,T],$ and using the following interpolation  inequality:
$$
\|Z\|_{\dot B^{1/2}_{2,1}}\leq C\|Z\|_{\dot B^{-1}_{2,1}}^{1/4}
\|\nabla Z\|_{\dot B^{0}_{2,1}}^{3/4},
$$
together with Young inequality, we may write for all $\eps>0,$
$$
\begin{aligned}
 \|V \cdot \nabla V\|_{L_1(0,T;\dot B^0_{2,1})} &\leq C\int_{0}^T \|V\|_{\dot B^{1/2}_{2,1}} \|\nabla V\|_{\dot B^{1/2}_{2,1}}\, dt \\
 &\leq C\int_{0}^T \|V\|_{\dot B^{1/2}_{2,1}} \, \|\nabla V\|^{1/4}_{\dot B^{-1}_{2,1}}\| \nabla^2 V\|^{3/4}_{\dot B^{0}_{2,1}}\,dt\\
  &\leq  \frac{C}{\eps^3\mu^3} \int_0^T \|V\|^4_{\dot B^{1/2}_{2,1}} \, \|V\|_{\dot B^{0}_{2,1}}\,dt + \eps\mu \|\nabla^2V\|_{L_1(0,T;\dot B^{0}_{2,1})}.\end{aligned}
  $$
Hence, reverting to \eqref{V5} and taking $\eps$ small enough, we find that
\begin{equation}\label{V8}
 \|V\|_{L_\infty(0,T;\dot B^0_{2,1})} + \|V_t,\mu\nabla^2V\|_{L_1(0,T;\dot B^{0}_{2,1})} \leq 
C \bigg( \frac{1}{\mu^3}\int_0^T \|V\|^4_{\dot B^{1/2}_{2,1}} \, \|V\|_{\dot B^{0}_{2,1}}\,dt + \|V_0\|_{\dot B^0_{2,1}}\bigg).
\end{equation}
In view of the Gronwall inequality,  this gives
\begin{equation}\label{V9}
 \|V\|_{L_\infty(0,T;\dot B^0_{2,1}(\R^2)} \leq C\|V_0\|_{\dot B^0_{2,1}(\R^2)}\exp\biggl( \frac C{\mu^3}\int_{0}^T \|V\|^4_{\dot B^{1/2}_{2,1}} dt\biggr)\cdotp
\end{equation}
Remembering \eqref{eq:interpo}, one can conclude to  \eqref{V10}.
\medbreak
For the sake of completeness, we have to prove that $\nabla^2V\in L_1(\R_+;\dot B^0_{2,1})$ as it has been assumed
implicitly in the above computations. This may be obtained by bootstrap 
from the property that $V\in L_4(\R_+;\dot B^{\frac12}_{2,1}).$  Indeed, because  $V\cdot\nabla V=\div(V\otimes V),$
 Lemma \ref{l:para1} gives $V\cdot\nabla V\in  L_2(\R_+;\dot B^{-1}_{2,1}),$ and thus
 $V\in L_2(\R_+;\dot B^1_{2,1})$ through \eqref{V-Stokes}, thanks to  maximal regularity results for the Stokes system. 
 Knowing that  $V\in L_2(\R_+;\dot B^1_{2,1}),$ Lemma \ref{l:para1} now gives us 
  $V\cdot\nabla V\in  L_1(\R_+;\dot B^{0}_{2,1}),$ and thus  $(V_t,\nabla^2V)\in  L_1(\R_+;\dot B^{0}_{2,1}).$
  \end{proof}
\bigbreak
\noindent{\bf Acknowledgments.}
The first author (R.D.)  is partly supported by ANR-15-CE40-0011.
The second author (P.B.M.) has been partly supported by National  Science  Centre  grant 2014/14/M/ST1/00108 (Harmonia).


\bigskip

\begin{thebibliography}{99}


\bibitem{BMN} A. Babin, A. Mahalov and B. Nicolaenko: Global regularity of 3D rotating Navier-Stokes equations for resonant domains, {\em Indiana Univ. Math. J.},  {\bf 48}(3) pages 1133--1176 (1999).

 \bibitem{BCD} H. Bahouri, J.-Y. Chemin and  R. Danchin: {\it Fourier Analysis and Nonlinear Partial Differential Equations,} Grundlehren der mathematischen Wissenschaften, {\bf 343}, Springer (2011).





\bibitem{Gal} J.-Y. Chemin, I. Gallagher and M. Paicu: Global regularity for some classes of large solutions to the Navier-Stokes equations. {\em Ann. of Math.} {\bf 173}{2} pages  983--1012 (2011).

\bibitem{D1} R. Danchin: Global existence in critical spaces for compressible
Navier-Stokes equations, {\em Inventiones Mathematicae}, {\bf 141}(3), pages 579--614 (2000).


 \bibitem{D3}  R. Danchin: Zero Mach number limit in critical spaces for compressible Navier-Stokes equations,
   {\em Ann. Sci. \'{E}cole Norm. Sup.}, {\bf 35}(1), pages 27--75 (2002).

 \bibitem{D4} R. Danchin: Zero Mach Number Limit for Compressible Flows with Periodic Boundary Conditions,
  {\em American  Journal of  Mathematics}, {\bf 124}(6),  pages 1153--1219 (2002).
    
  \bibitem{D5}   R. Danchin: Well-posedness in critical spaces for
barotropic viscous fluids with truly not constant density, {\em Communications in Partial Differential Equations}, 
  {\bf 32}, pages 1373--1397 (2007).
  


\bibitem{DG} B. Desjardins and E. Grenier:  Low Mach number limit of viscous compressible flows in the whole space, {\em Proc. Roy. Soc. London Ser. A, Math. Phys. Eng. Sci.}, {\bf 455}, pages  2271--2279 (1999).

\bibitem{FN-book} E. Feireisl and A. Novotn{\'y}: {\em Singular limits in thermodynamics of viscous fluids},
  Advances in Mathematical Fluid Mechanics, Birkh\"auser Verlag, Basel (2009).

\bibitem{GIP} I. Gallagher, D. Iftimie and F. Planchon: Asymptotics and stability for global solutions to the NavierÐStokes equations, {\em Ann. Inst. Fourier}, , {\bf 53},  2003, pages 1387--1424.


\bibitem{Haspot3} B. Haspot:
 Global existence of strong solution for shallow water system with large initial data on the irrotational part,  arXiv:1201.5456.


\bibitem{VK}   A.V. Kazhikhov and V.A. Vaigant:  On the existence of global solutions of two-dimensional Navier-Stokes equations of a compressible viscous fluid, {\em  Siberian Math. J.}, {\bf 36}(6)  pages 1108--1141 (1995).
 
\bibitem{KS}  T. Kobayashi and Y. Shibata:  Decay estimates of solutions for the equations of motion of compressible viscous and heat-conductive gases in an exterior domain in $\R^3,$
{\em Comm. Math. Phys.}, {\bf 200}(3), pages  621--659 (1999). 

\bibitem{Kot1} M. Kotschote:  Strong solutions to the compressible non-isothermal Navier-Stokes equations, {\em Adv. Math. Sci. Appl.}, {\bf 22}(2), pages  319--347 (2012).

\bibitem{Kot2} M. Kotschote: Dynamical Stability of Non-Constant Equilibria for the Compressible NavierÐStokes Equations in Eulerian Coordinates, {\em Communications in Mathematical Physics}, {\bf 328}(2), pages 809--847 (2014).

\bibitem{OL} O.A. Ladyzhenskaya:  Solution `in the large' of the non-stationary boundary value problem for the Navier-Stokes system with two space variables, {\em Comm. Pure Appl. Math.}, {\bf  12} pages  427--433 (1959).

\bibitem{Leray} J. Leray: Essai sur le mouvement d'un liquide visqueux emplissant l'espace, {\em Journal de Math\'ematiques Pures et Appliqu\'ees}, {\bf 13}, pages 331--418 (1934).



\bibitem{MN80} A. Matsumura and T. Nishida: The initial value problem for equations of motion of compressible viscous and heat conductive gases, {\em J. Math. Kyoto Univ.}, {\bf 20}, pages  67--104 (1980).

\bibitem{MN83} A. Matsumura and T. Nishida: Initial boundary value problem for equations of motion of compressible viscous and heat conductive fluids, {\em Commun. Math. Phys.}, {\bf  89}, pages 445--464 (1983).

\bibitem{M01} P.B. Mucha: Stability of nontrivial solutions of the Navier-Stokes system on the three dimensional torus, 
{\em J. Differential Equations}, {\bf 172}(2), pages  359--375 (2001).

\bibitem{M} P.B. Mucha: The Cauchy problem for the compressible Navier-Stokes equations in
   the $L_p$-framework,  \emph{Nonlinear Anal.}, {\bf 52}(4), pages 1379--1392 (2003). 
   

\bibitem{M08} P.B. Mucha:  Stability of 2D incompressible flows in $\R^3,$ {\em  J. Differential Equations}, 
{\bf  245}(9), pages  2355--2367 (2008).

\bibitem{MZ02} P.B. Mucha and  W. Zaj\c aczkowski:    On a $L_p$-estimate for the linearized compressible Navier-Stokes equations with the Dirichlet boundary conditions, {\em J. Differential Equations}, {\bf 186}, pages 377--393 (2002). 

    \bibitem{MZ} P.B. Mucha and W. Zaj\c aczkowski: Global existence of solutions of the Dirichlet problem for the compressible Navier-Stokes equations,
 {\em Z. Angew. Math. Mech.}, {\bf 84}, 417--424 (2004).

\bibitem{PRST}
G. Ponce, R. Racke, T.C. Sideris and  E.S. Titi: Global stability of large solutions to the 3D Navier-Stokes equations, 
{\em Comm. Math. Phys.}, {\bf 159}, pages 329--341 (1994).

\bibitem{VZ} A. Valli and  W. Zaj\c aczkowski:   Navier-Stokes equations for compressible fluids: global existence and qualitative properties of the solutions in the general case, {\em Com. Math. Phys.}, {\bf 103}, pages  259--296 (1986).

\end{thebibliography}
\end{document}